\documentclass[12pt]{article}
\usepackage{amssymb}
\usepackage{latexsym}
\usepackage{amsthm}
\usepackage{amscd}
\usepackage{amsmath}
\usepackage{diagrams}
\usepackage{mathrsfs}
\usepackage{amsfonts}

\usepackage{a4wide}
\usepackage{amsmath}
\usepackage{amssymb}
\usepackage{amsthm}
\usepackage{latexsym}
\usepackage{graphicx}
\usepackage[english]{babel}
\usepackage{makeidx}

\newtheorem{thm}{Theorem}[section]

\newtheorem{defn-lem}[thm]{Definition-Lemma}

\newtheorem{ex}[thm]{Example}
\newtheorem{prop}[thm]{Proposition}

\theoremstyle{remark}
\newtheorem{rem}[thm]{Remark}

\theoremstyle{definition}

\numberwithin{equation}{section}

\allowdisplaybreaks[4]
\def \CA{{\mathcal A}}

\def \C{{\mathbb C}}

\def\map#1.#2.{#1 \longrightarrow #2}
\def\rmap#1.#2.{#1 \dasharrow #2}
\DeclareMathOperator{\rank}{rank}

\DeclareMathOperator{\corank}{corank}

\def\fb#1.{\underset #1 \to \times}
\def\pr#1.{\Bbb P^{#1}}
\def\ring#1.{\mathcal O_{#1}}
\def\mlist#1.#2.{{#1}_1,{#1}_2,\dots,{#1}_{#2}}

\def\uloopr#1{\ar@'{@+{[0,0]+(-4,5)} @+{[0,0]+(0,10)}
@+{[0,0]+(4,5)}}^{#1}}

\def\dloopr#1{\ar@'{@+{[0,0]+(-4,-5)} @+{[0,0]+(0,-10)}
@+{[0,0]+(4,-5)}}_{#1}}

\def\rloopd#1{\ar@'{@+{[0,0]+(5,4)} @+{[0,0]+(10,0)}
@+{[0,0]+(5,-4)}}^{#1}}

\newcommand{\rdots}{\mathinner{\mkern1mu\raise1pt\hbox{.}\mkern2mu\raise4pt\hbox{.} \mkern2mu\raise7pt\vbox{\kern7pt\hbox{.}}\mkern1mu}} 

\def\lloopd#1{\ar@'{@+{[0,0]+(-5,4)} @+{[0,0]+(-10,0)}
@+{[0,0]+(-5,-4)}}_{#1}}

\long\def\ignore#1{}
\long\def\ignore#1{#1}

\begin{document}

\begin{center}
{\bf A Short Note on the Thom-Boardman Symbols of Differentiable Maps}

\bigskip

\end{center}

\begin{center}  
Yulan Wang, Jiayuan Lin and Maorong Ge
\end{center}

{\small {\bf Abstract} It is well known that Thom-Boardman symbols are realized by non-increasing sequences of nonnegative integers. A natural question is whether the converse is also true. In this paper we answer this question affirmatively, that is, for any non-increasing sequence of nonnegative integers, there is a map-germ with the prescribed sequence as its Thom-Boardman symbol.}

\section{Introduction}

Thom-Boardman symbols were first introduced by R. Thom and later generalized by J. M. Boardman to classify singularities of differentiable maps. They are realized by non-increasing sequences of nonnegative integers. Although Thom-Boardman symbols have been around for over $50$ years, in general to compute those numbers is extremely difficult. Before J. Lin and J. Wethington \cite{LinWeth} proved R. Varley's conjecture on the Thom-Boardman symbols of polynomial multiplication maps, there were only sporadic known results. J. Lin and J. Wethington \cite{LinWeth} provides infinitely many examples of map-germs with distinct Thom-Boardman symbols. Since Thom-Boardman symbol are given by non-increasing sequences of nonnegative integers, a natural question is whether the converse is also true. In this paper we answer this question affirmatively. We prove that for any given non-increasing sequence of nonnegative integers, there is a map-germ with the prescribed sequence as its Thom-Boardman symbol.

For the reader's convenience, let us briefly recall the definition of Thom-Boardman symbols from \cite{Arnold}.

Let $x_1, \cdots, x_m$ be local coordinates on a differential manifold $M$ of dimension $m$. Denote $\CA$ the local ring of germs of differentiable functions at a point $x \in M$. For any ideal $B$ in $\CA$, the $\textit{Jacobian extension}$, $\Delta_k B$, is the ideal spanned by $B$ and all the minors of order $k$ of the $\textit{Jacobian matrix}$ $(\partial \phi_i/ \partial x_j)$, denoted $\delta B$, formed from partial derivatives of functions $\phi_i$ in $B$. We say that $\Delta_i B$ is $\textit{critical}$ if $\Delta_i B \ne \CA$ but $\Delta_{i-1} B = \CA$ (just $\Delta_1 B \ne \CA$ when $i=1$). That is, the $\textit{critical extension}$ of $B$ is $B$ adjoined with the least order minors of the Jacobian matrix of $B$ for which the extension does not coincide with the whole algebra. 

Suppose that $N$ is another differential manifold of dimension $n$ and $y_1, \cdots, y_n$ be local coordinates on it. For a differential map $F: M \rightarrow N, F=(f_1,f_2, \cdots, f_n)$, we denote $J$ the ideal generated by $f_1, \cdots, f_n$ in $\CA$. Then $\Delta_k J$ is spanned by $J$ and all the minors of order $k$ of the $\textit{Jacobian matrix}$ $\delta J=(\partial f_i/ \partial x_j)$.

Now we shift the lower indices to upper indices of the critical extensions by the rule $\Delta^i J=\Delta_{m-i+1} J$. We repeat the process described above with the resulting ideals until we have a sequence of critical extensions of $J$,

$$J \subseteq \Delta^{i_1} J \subseteq \Delta^{i_2} \Delta^{i_1} J \subseteq \cdots \subseteq \Delta^{i_k} \Delta^{i_{k-1}} \cdots \Delta^{i_1} J \subseteq \cdots.$$

The non-increasing sequence $(i_1, i_2, \cdots, i_k, \cdots)$ is called the $\textit{Thom-Boardman symbol}$ of $J$, denote $TB(J)$. The purpose of switching the indices is that doing so allows us to express $TB(J)$ as follows:

$$i_1=\corank (J), i_2= \corank (\Delta^{i_1} J), \cdots, i_k= \corank (\Delta^{i_{k-1}} \cdots \Delta^{i_1} J), \cdots$$

\noindent where the $\rank$ of ideal is defined to be the maximal number of independent coordinates from the ideal and the $\corank$ is the number of variables minus the rank. 

We also need the following construction from \cite{LinWeth}.

Let $M_n$ be the set of monic complex polynomials in one variable of degree $n$. $M_n \cong \C^n$ by the map sending 
$f(x)=x^n+a_{n-1} x^{n-1}+ \cdots + a_0$ to the $n$-tuple $(a_0,a_1,\cdots, a_{n-1}) \in \C^n$.

If we take $f(x)$ of degree $n$ as above and $g(x)=x^r+ b_{r-1} x^{r-1} + \cdots + b_0$ of degree $r$, then the product $h(x)=f(x) g(x)$ is a monic polynomial of the form $h(x)=x^{n+r} + c_{n+r-1} x^{n+r-1} + \cdots + c_0$, where the $c_j$'s are polynomials in the coefficients of $f$ and $g$. This gives us maps

$$\mu_{n,r}: \C^n \times \C^r \rightarrow \C^{n+r}$$

\noindent defined by 

$$(a_0, \cdots, a_{n-1}, b_0, \cdots, b_{r-1}) \rightarrow (c_{n+r-1}, \cdots, c_0).$$

\noindent Assume $n \ge r$ and consider the Euclidean algorithm applied to $n$ and $r$:

$$\begin{array}{lr} n=q_1 r+r_1,  \hskip .44 cm 0 < r_1 < r\\
r=q_2 r_1+r_2,  \hskip .35 cm 0 < r_2 < r_1 \\
\vdots\\
r_{k-1}=q_{k+1} r_k, \hskip .3 cm  0 < r_k <r_{k-1}.\end{array}$$

\noindent Let $I(n,r)$ be the tuple given by the Euclidean algorithm on $n$ and $r$:

$$I(n,r)=(r, \cdots, r, r_1, \cdots, r_1, \cdots, r_k, \cdots, r_k, 0, \cdots)$$

\noindent where $r$ is repeated $q_1$ times, and $r_i$ is repeated $q_{i+1}$ times.

\medskip

Let $I(\mu_{n,r})$ be the ideal in the algebra $\CA$ of germs at origin generated by $c_{j}$'s in the map $\mu_{n,r}: \C^n \times \C^r \rightarrow \C^{n+r}$. Denote $TB(I(\mu_{n,r}))$ the Thom-Boardman symbol of this ideal, Robert Varley conjectured that $TB(I(\mu_{n,r}))=I(n,r)$ for any $n \ge r$. In \cite{LinWeth}, J. Lin and J. Wethington confirmed Varley's Conjecture. For the reader's convenience, let us state their result here.

\begin{thm}
$TB(I(\mu_{n,r}))=I(n,r)$ for any $n \ge r$.
\end{thm}

Professor J. M. Boardman provides us the following example which has constant Thom-Boardman symbol.

\begin{ex} The zero map-germ at origin $F: \C^a \rightarrow \C^b$ has Thom-Boardman symbol $(a, a, \cdots, a, \cdots )$. 
\end{ex}

Denote $\mu_{\infty,a}$ the map-germ in Example $1.2$. We have the following theorem.

\begin{thm} Let $(i_0, \cdots, i_0, i_1, \cdots, i_1, \cdots, i_k, \cdots, i_k, \cdots)$ be a non-increasing sequence of nonnegative integers, where $i_0 >i_1 > \cdots > i_k \ge 0$ for some $k$, $i_j$ repeats $l_j$ times for each $j <k$ and $i_k$ repeats infinitely many times. Denote $\mu$ as the Cartesian product of the map-germs of 

$$\mu_{(i_0-i_1) l_0,i_0-i_1} \times \mu_{(i_1-i_2) (l_0+l_1),i_1-i_2} \times \cdots \times \mu_{(i_{k-1}-i_k) (l_0+l_1+\cdots+l_{k-1}),i_{k-1}-i_k} \times \mu_{\infty,i_k}$$

\noindent at origin. Then $(i_0, \cdots, i_0, i_1, \cdots, i_1, \cdots, i_k, \cdots, i_k, \cdots)$ is the Thom-Boardman symbol of the map-germ $\mu$.

\end{thm}

\begin{rem}
Historically there were only sporadic known Thom-Boardman symbols for some map-germs. J. Lin and J. Wethington \cite{LinWeth} gave infinitely many examples of map-germs with distinct Thom-Boardman symbols.  However, Theorem $1.3$ provides us a complete list of representatives of map-germs classified by their Thom-Boardman symbols. We should warn the reader that the classification of map-germs by their Thom-Boardman symbols is not complete. Two non- left-right equivalent map-germs may have the same Thom-Boardman symbols. Under such a circumstance, other invariants are needed (e.g. see page $67$ in \cite{Arnold} and \cite{Adams}). 
\end{rem}

The proof of Theorem $1.3$ is very simple. We first prove that Thom-Boardman symbols have additive property under Cartesian product. Combining this property with Theorem $1.1$, we immediately obtain Theorem $1.3$. 

{\bf Acknowledgments} The authors thank Professor J. M. Boardman for providing us Example $1.2$, which is even not well-known among experts. We are grateful for his generosity to allow us to use his example in Theorem $1.3$. Without it, Theorem $1.3$ would be incomplete.

\section{Proof of Theorem $1.3$}

We first prove that Thom-Boardman symbols have additive property under Cartesian product.

\begin{prop}
Let $F_t: M_t \rightarrow N_t, t=1,2$ be two differentiable maps. Then the Thom-Boardman symbol $TB(F)$ of $F=(F_1,F_2): M_1 \times M_2 \rightarrow N_1 \times N_2$ is equal to the sum of $TB(F_1)$ and $TB(F_2)$, where $F=(F_1,F_2): M_1 \times M_2 \rightarrow N_1 \times N_2$  is the Cartesian product of $F_t: M_t \rightarrow N_t, t=1,2$.
\end{prop}

\begin{proof}
Denote $TB(F)=(s_1, s_2, \cdots, s_p, \cdots)$ and $TB(F_t)=(s_1^{(t)}, s_2^{(t)}, \cdots, s_p^{(t)}, \cdots), t=1,2$. We will show that $s_p=s_p^{(1)}+s_p^{(2)}$ for all integer $p \ge 1$.

Let $x_1^{(t)}, \cdots, x_{m_t}^{(t)}$ be local coordinates of origin on the differential manifolds $M_t, t=1,2$ and $y_1^{(t)}, \cdots, y_{n_t}^{(t)}$ be local coordinates on the differential manifolds $N_t, t=1,2$. Denote $F_t: M_t \rightarrow N_t, F_t=(f_1^{(t)}, \cdots, f_{n_t}^{(t)})$ and $J_{F_t}$ the ideal generated by $f_1^{(t)}, \cdots, f_{n_t}^{(t)}$ in $\CA^{(t)}, t=1,2$, the local ring of germs of differentiable functions at origin in $M_t$. Denote $J_F$ the ideal generated by $f_1^{(1)}, \cdots, f_{n_1}^{(1)}; f_1^{(2)}, \cdots, f_{n_2}^{(2)}$ in $\CA$, the local ring of germs of differentiable functions at origin in $M_1 \times M_2$. Note that in general $\CA^{(1)} \times \CA^{(2)} \varsubsetneq \CA $ and $J_{F_1} \times J_{F_2} \varsubsetneq J_F $. However, it does not affect our conclusion on the additive property of Thom-Boardman symbols. Using induction, we will prove that $\Delta^{s_p} \Delta^{s_{p-1}} \cdots \Delta^{s_1} J_F=(\Delta^{s_p^{(1)}} \Delta^{s_{p-1}^{(1)}} \cdots \Delta^{s_1^{(1)}} J_{F_1}; \Delta^{s_p^{(2)}} \Delta^{s_{p-1}^{(2)}} \cdots \Delta^{s_1^{(2)}} J_{F_2})$ in $\CA $, which immediately implies the conclusion in Proposition $2.1$.

When $p=1$, it is easy to see that the Jacobian matrix 

\begin{equation}
\begin{split}
\delta J_F=
\begin{pmatrix} 
\delta J_{F_1} &0\cr
0 &\delta J_{F_2}\cr
\end{pmatrix}. 
\end{split}
\end{equation}

Therefore $\corank (J_F)=\corank (J_{F_1})+\corank (J_{F_2})$, that is, $s_1=s_1^{(1)}+s_1^{(2)}$.

The critical extension of $\Delta^{s_1} J_F$ is spanned by $f_1^{(1)}, \cdots, f_{n_1}^{(1)}; f_1^{(2)}, \cdots, f_{n_2}^{(2)}$ and all the minors of order $m_1+m_2-s_1+1$ of the Jacobian matrix $\delta J_F$. Any minor of order $m_1+m_2-s_1+1$ of the Jacobian matrix $\delta J_F$ contains either at least $m_1-s_1^{(1)}+1$ rows from $\delta J_{F_1}$ or $m_2-s_1^{(2)}+1$ rows from $\delta J_{F_2}$, so $\Delta^{s_1} J_F \subseteq (\Delta^{s_1^{(1)}} J_{F_1}; \Delta^{s_1^{(2)}} J_{F_2})$ in $\CA$. It is easy to see that the inclusion $\Delta^{s_1} J_F \subseteq (\Delta^{s_1^{(1)}} J_{F_1}; \Delta^{s_1^{(2)}} J_{F_2})$ is actually an equality because any minor of order $m_t-s_1^{(t)}+1$ of $\delta J_{F_t}, t=1, 2$ is only different from a minor of order $m_1+m_2-s_1+1$ of $\delta J_F$ by a unit in $\CA$. 

Suppose that $\Delta^{s_p} \Delta^{s_{p-1}} \cdots \Delta^{s_1} J_F=(\Delta^{s_p^{(1)}} \Delta^{s_{p-1}^{(1)}} \cdots \Delta^{s_1^{(1)}} J_{F_1}; \Delta^{s_p^{(2)}} \Delta^{s_{p-1}^{(2)}} \cdots \Delta^{s_1^{(2)}} J_{F_2})$ in $\CA $ for some $p \ge 1$. Because the generators of $\Delta^{s_p^{(1)}} \Delta^{s_{p-1}^{(1)}} \cdots \Delta^{s_1^{(1)}} J_{F_1}$ involve only $x_1^{(1)}, \cdots, x_{m_1}^{(1)}$ and the generators of $\Delta^{s_p^{(2)}} \Delta^{s_{p-1}^{(2)}} \cdots \Delta^{s_1^{(2)}} J_{F_2}$ involve only $x_1^{(2)}, \cdots, x_{m_2}^{(2)}$, the Jacobian matrix $\delta \Delta^{s_p} \Delta^{s_{p-1}} \cdots \Delta^{s_1} J_F$ has the form

\begin{equation}
\begin{split}
\delta \Delta^{s_p} \Delta^{s_{p-1}} \cdots \Delta^{s_1} J_F=
\begin{pmatrix} 
\delta \Delta^{s_p^{(1)}} \Delta^{s_{p-1}^{(1)}} \cdots \Delta^{s_1^{(1)}} J_{F_1}&0\cr
0 &\delta \Delta^{s_p^{(2)}} \Delta^{s_{p-1}^{(2)}} \cdots \Delta^{s_1^{(2)}} J_{F_2}\cr
\end{pmatrix}
\end{split}
\end{equation}

The same argument as the case $p=1$ gives $s_{p+1}=s_{p+1}^{(1)}+s_{p+1}^{(2)}$ and $\Delta^{s_{p+1}} \Delta^{s_p}  \cdots \Delta^{s_1} J_F$

\noindent $=(\Delta^{s_{p+1}^{(1)}} \Delta^{s_p^{(1)}}  \cdots \Delta^{s_1^{(1)}} J_{F_1}; \Delta^{s_{p+1}^{(2)}} \Delta^{s_p^{(2)}}  \cdots \Delta^{s_1^{(2)}} J_{F_2})$ in $\CA $. This completes the proof of Proposition $2.1$.
\end{proof}

The proof of Theorem $1.3$ is an easy corollary of Proposition $2.1$ and Theorem $1.1$. 

\begin{proof} (Proof of Theorem $1.3$)

We use induction on $k$ to prove Theorem $1.3$.

When $k=0$, the sequence in Theorem $1.3$ is a constant sequence, it is the Thom-Boardman symbol of the map-germ $\mu_{\infty,i_0}$.

Suppose that Theorem $1.3$ is true for $k=p$, $p \ge 0$. When $k=p+1$, by inductive assumption and Theorem $1.1$, the map-germ 

\noindent $\tilde{\mu}= \mu_{[(i_0-i_{p+1})-(i_1-i_{p+1})] l_0,[(i_0-i_{p+1})-(i_1-i_{p+1})]} \times \mu_{[(i_1-i_{p+1})-(i_2-i_{p+1})] (l_0+l_1),[(i_1-i_{p+1})-(i_2-i_{p+1})]} \times \cdots \times \mu_{[(i_{p}-i_{p+1})-(i_{p+1}-i_{p+1})] (l_0+l_1+\cdots+l_{p}),[(i_{p}-i_{p+1})-(i_{p+1}-i_{p+1})]}=\mu_{(i_0-i_1) l_0,i_0-i_1} \times \mu_{(i_1-i_2) (l_0+l_1),i_1-i_2}$

\noindent $\times \cdots \times \mu_{(i_{p}-i_{p+1}) (l_0+l_1+\cdots+l_{p}),i_{p}-i_{p+1}}$ has Thom-Boardman symbol $(i_0-i_{p+1}, \cdots, i_0-i_{p+1}, i_1-i_{p+1}, \cdots, i_1-i_{p+1}, \cdots, i_p-i_{p+1}, \cdots, i_p-i_{p+1}, 0, \cdots, 0, \cdots)$. Now let $\mu=\tilde{\mu} \times \mu_{\infty,i_{p+1}}$. By Proposition $2.1$, it has Thom-Boardman symbol $(i_0, \cdots, i_0, i_1, \cdots, i_1, \cdots, i_{p+1}, \cdots, i_{p+1}, \cdots$

\noindent $)$. This completes the proof of Theorem $1.3$.

\end{proof} 

\begin{rem}
As an easy consequence of Proposition $2.1$ and Theorem $1.1$, we also obtain that the map-germs $\mu_{kr,r}$ and Cartesian product of $r$ copies of $\mu_{k,1}$ have the same Thom-Boardman symbols. So Theorem $1.3$ can be restated in term of $\mu_{k,1}$'s and $\mu_{\infty,1}$. Those special map-germs form building blocks for map-germs with arbitrary Thom-Boardman symbols.
\end{rem}

{}

{\footnotesize DEPARTMENT OF MATHEMATICS, SUNY CANTON, 34 CORNELL DRIVE, CANTON, 

\hskip .05 cm NY 13617, USA}

{\footnotesize ANHUI ECONOMIC MANAGEMENT INSTITUTE, HEFEI, ANHUI, 230059, CHINA}

$\textit{E-mail address:}$ wangy@canton.edu

\medskip

{\footnotesize DEPARTMENT OF MATHEMATICS, SUNY CANTON, 34 CORNELL DRIVE, CANTON, 

\hskip .05 cm NY 13617, USA}

$\textit{E-mail address:}$ linj@canton.edu

\medskip

{\footnotesize DEPARTMENT OF MATHEMATICS, ANHUI UNIVERSITY, HEFEI, ANHUI, 230039, CHINA}

$\textit{E-mail address:}$ ge1968@126.com
\end{document}